\documentclass[reqno,a4paper,final]{amsart}

\usepackage[utf8]{inputenc}
\usepackage[T1]{fontenc}
\usepackage[english]{babel}
\usepackage{ifthen}
\usepackage{cite}
\usepackage{enumitem}
\usepackage{amsmath}
\usepackage{amsthm}
\usepackage{amsfonts}
\usepackage{amssymb}
\usepackage{mathtools}
\usepackage[bookmarksopen]{hyperref}

\numberwithin{equation}{section} 

\newtheorem{lemma}{Lemma}[section]
\newtheorem{corollary}[lemma]{Corollary}
\newtheorem{theorem}[lemma]{Theorem}

\theoremstyle{definition}

\newtheorem{remark}[lemma]{Remark}

\newlist{thm_enum}{enumerate}{1}
\setlist[thm_enum]{label=\normalfont(\alph*)}
\newlist{def_enum}{enumerate}{1}
\setlist[def_enum]{label=\normalfont(\roman*)}

\newcommand{\IN}{\mathbb{N}}
\newcommand{\IR}{\mathbb{R}}
\newcommand{\IC}{\mathbb{C}}

\renewcommand{\epsilon}{\varepsilon}
\renewcommand{\phi}{\varphi}
\newcommand{\abs}[1]{\left\lvert#1\right\rvert}
\newcommand{\latticenorm}[1]{\left\lvert#1\right\rvert}
\newcommand{\norm}[1]{\left\lVert#1\right\rVert}

\newcommand{\R}[2][\empty]{
	\ifthenelse{\equal{#1}{\empty}}
		{\mathcal{R}\left\{#2\right\}}
		{\mathcal{R}_{#1}\left\{#2\right\}}
}

\DeclareMathOperator{\Id}{Id}
\DeclareMathOperator{\sign}{sign}

\DeclareMathOperator{\Cos}{Cos}
\DeclareMathOperator{\Rad}{Rad}
\DeclareMathOperator{\Mor}{Mor}

\allowdisplaybreaks[4]

\begin{document}

	\title{Regularity of Semigroups via the Asymptotic Behaviour at Zero}
	\begin{abstract}
		An interesting result by T. Kato and A. Pazy says that a contractive semigroup $(T(t))_{t \ge 0}$ on a uniformly convex space $X$ is holomorphic iff $\limsup_{t \downarrow 0} \norm{T(t) - \Id} < 2$. We study extensions of this result which are valid on arbitrary Banach spaces for semigroups which are not necessarily contractive. This allows us to prove a general extrapolation result for holomorphy of semigroups on interpolation spaces of exponent $\theta \in (0,1)$. As application we characterize boundedness of the generator of a cosine family on a UMD-space by a zero-two law. Moreover, our methods can be applied to $\mathcal{R}$-sectoriality: We obtain a characterization of maximal regularity by the behaviour of the semigroup at zero and show extrapolation results.
		
	\end{abstract}  
	
	\author{Stephan Fackler}
	\address{Institute of Applied Analysis, University of Ulm, Helmholtzstr. 18, 89069 Ulm}
	\email{stephan.fackler@uni-ulm.de}
	\thanks{The author was supported by a scholarship of the ``Landesgraduiertenförderung Baden-Württemberg''}
	
	\maketitle
	
	\section{Introduction}
	In 1970, T. Kato showed that a strongly continuous semigroup $(T(t))_{t \ge 0}$ on a Banach space $X$ is holomorphic (i.e. $T$ has a holomorphic extension to a sector $\Sigma \coloneqq \{ z \in \IC \setminus \{ 0 \}: \abs{\arg z} < \delta \}$ for some $\delta \in (0, \frac{\pi}{2}])$ if $\limsup_{t \downarrow 0} \norm{T(t) - \Id} < 2$~\cite{Kat70}. A similar weaker result was obtained by J.W. Neuberger in~\cite[Theorem~A]{Neu70}. A. Pazy showed that the converse is valid if $(T(t))_{t \ge 0}$ is assumed to be contractive and $X$ is assumed to be uniformly convex~\cite[Corollary~5.8]{Paz83}. Thus, a regularity property (viz. holomorphy) can be characterized by the asymptotic behaviour of the semigroup as $t \downarrow 0$.
	
	The purpose of this article is to extend these results in various ways. Each of our results allows us to characterize a regularity property, namely holomorphy, maximal regularity and uniform continuity, by an asymptotic behaviour for small time.
	
	T. Kato and A. Pazy's criterion is not fully satisfying as it does not completely characterize holomorphy. It is not even invariant under equivalent renorming of the underlying space. Our first goal is to prove a version of the above result which is valid on arbitrary Banach spaces and for arbitrary (not necessarily contractive) strongly continuous semigroups. We show that a $C_0$-semigroup $(T(t))_{t \ge 0}$ on a Banach space $X$ is holomorphic if and only if there is a polynomial $f \in \IC[z]$ with
		\begin{equation}
			\limsup_{t \downarrow 0} \norm{f(T(t))} < \norm{f}_{\mathbb{D}} \label{eq:characterization},
		\end{equation}
	where $\norm{f}_{\mathbb{D}} \coloneqq \sup_{\abs{z} \le 1} \abs{f(z)}$ is the norm of $f$ in the disc algebra. The criterion of Kato-Pazy is the special case where $f(z) = z-1$.
	A characterization similar to~\eqref{eq:characterization} was already obtained by A. Beurling in~\cite[Theorem~III]{beu70}. We will therefore call results of the above type \emph{Kato-Beurling} theorems. However, Beurling's original proof is very long and technical, whereas we give a short proof based on Kato's ideas. Kato's method was extended in a different direction by C.J.K. Batty and S. Król in~\cite{Kro09} and~\cite{BatKro10} to obtain new perturbation results for holomorphy and other regularity properties of semigroups. 
	
	As a direct consequence of our characterization we prove a very general extrapolation result for holomorphy of semigroups. Given two consistent semigroups $(T_1(t))_{t \ge 0}$ and $(T_2(t))_{t \ge 0}$ on an interpolation couple $(X_1, X_2)$ and a regular interpolation space $X$ of exponent $\theta \in (0,1)$ with respect to $(X_1, X_2)$, we obtain the following: if $(T_1(t))_{t \ge 0}$ is holomorphic, then the interpolated semigroup $(T(t))_{t \ge 0}$ on $X$ is holomorphic as well. In other words, holomorphy is a property which \emph{extrapolates}.
	In the second part we use the obtained characterization to prove a zero-two law for cosine families. More precisely, given a strongly continuous cosine family $(\Cos(t))_{t \in \IR}$ on a UMD-space we show that 
		\begin{equation*}
			\limsup_{t \downarrow 0} \norm{\Cos(t) - \Id} < 2
		\end{equation*}
implies the boundedness of the generator of the cosine family which in turn is equivalent to the continuity of $\Cos(\cdot)$ at zero for the operator norm.

	In the third part we apply the developed methods to $\mathcal{R}$-boundedness and maximal regularity. The characterization of maximal regularity of the generator $A$ of a holomorphic semigroup $(T(z))_{z \in \Sigma}$ by the $\mathcal{R}$-analyticity of the generated semigroup obtained by L.Weis~\cite{Wei01} allows us - by extending the above theorems to the $\mathcal{R}$-analytic case - to prove a similar characterization of maximal regularity by the behaviour of $T(t)$ as $t \downarrow 0$. From this characterization we deduce extrapolation theorems for $\mathcal{R}$-analytic semigroups on interpolation spaces of exponent $\theta \in (0,1)$. We conclude with some remarks on dilation arguments.
		
%
%

The roadmap will be the following: in the first sections we formulate and prove our results in the holomorphic case in order to minimize technical issues and then present the generalizations to the $\mathcal{R}$-analytic case.
		
	\section{A short proof of the Kato-Beurling theorem}
	
	The following result is due to T. Kato~\cite{Kat70}. We present a short proof because of its beauty and importance in this paper. Further, the proof usually found in the literature~\cite[Theorem~5.6]{Paz83} that uses the same ideas is a bit longer.
	
	\begin{lemma}[Kato]\label{lem:kato} Let $(T(t))_{t \ge 0}$ be a strongly continuous semigroup. Assume that there exist constants $\abs{\zeta} = 1$, $t_0 > 0$ and $K > 0$ such that
		\begin{equation} \zeta \in \rho(T(t)) \quad \text{and} \quad \norm{(\zeta - T(t))^{-1}} \le K \qquad \text{ for all } 0 < t \le t_0. \end{equation}
		Then $(T(t))_{t \ge 0}$ extends to a holomorphic semigroup.
	\end{lemma}
	\begin{proof} Choose $\theta > 0$ such that $e^{i \theta} = \zeta$ and $\alpha$ such that $t \alpha = \theta$. Then for $\alpha \ge \frac{\theta}{t_0}$ we have $t \le t_0$. Further 
			\begin{equation*} e^{-it \alpha}T(t)x - x = (A - i\alpha) \int_0^t e^{-is \alpha} T(s)x \, ds = \int_0^t e^{-is \alpha} T(s) (A - i\alpha)x \, ds, \end{equation*}
		where the first equality holds for all $x \in X$ and the second for all $x \in D(A)$. Hence, $A - i\alpha$ is invertible and
			\begin{equation*} (A - i\alpha)^{-1} = -e^{it \alpha}(\zeta - T(t))^{-1} \int_0^t e^{-is \alpha} T(s) \, ds. \end{equation*}
		Choose $M$ such that $\norm{T(\cdot)}$ is bounded by $M$ on $[0,t_0]$. Then $\norm{(A - i\alpha)^{-1}} \le KMt = KM\theta \alpha^{-1}$. The same argument works for negative values of $\alpha$ if one replaces $\theta$ by $\tilde{\theta} < 0$ with $e^{i\tilde{\theta}} = \zeta$. Hence, there exist $\alpha_0 > 0$ and $C > 0$ such that $\{ i\alpha: \abs{\alpha} > \alpha_0 \} \subseteq \rho(A)$ and that $\norm{\alpha(A - i\alpha)^{-1}} \le C$ for all $\abs{\alpha} > \alpha_0$. By \cite[Corollary~3.7.18]{ABHN11}, this implies the holomorphy of the semigroup.
	\end{proof}
	
	
	We can directly apply Lemma~\ref{lem:kato} to prove a Kato-Beurling theorem for $C_0$-semigroups. We omit the proof of the converse implication which can be found in~\cite[p.~398]{beu70} and we will only present its generalization to the $\mathcal{R}$-analytic case later. Here and later $\IC_1[z]$ is the set of all polynomials $f \in \IC[z]$ with $\abs{f(1)} < \norm{f}_{\mathbb{D}}$.
	
	\begin{theorem}[Kato-Beurling]\label{thm:beurling} 
		Let $(T(t))_{t \ge 0}$ be a strongly continuous semigroup on a Banach space $X$. 
		\begin{thm_enum}
			\item\label{thm:beurling_a} $(T(t))_{t \ge 0}$ extends to a holomorphic semigroup if there is $f \in \IC[z]$ such that
				\begin{equation*} \limsup_{t \downarrow 0} \norm{f(T(t))} < \norm{f}_{\mathbb{D}}. \end{equation*}
			\item Conversely, if $(T(z))_{z \in \Sigma}$ is holomorphic, then for each $f \in \IC_1[z]$ there exists $N_0 \in \IN$ such that for every $N \ge N_0$ there exists $K_0 \in \IR_+$ with
				\[ \limsup_{t \downarrow 0} \norm{f^N(T(t))T(Kt)} < \norm{f^N}_{\mathbb{D}} \quad \text{for all } K \ge K_0. \]
		\end{thm_enum}
	\end{theorem}
	\begin{proof}
		Notice that by the assumptions $f$ cannot be constant. Further, as a consequence of the maximum principle we may assume that $\norm{f}_{\mathbb{D}} = f(\zeta) = 1$ for some $\zeta \in \partial \mathbb{D}$ (replace $f$ by $\tilde{f}(z) = e^{-i \arg f(\zeta)} \norm{f}_{\mathbb{D}}^{-1} f(z)$ if necessary). By assumption, there exist constants $t_0 > 0$ and $0 < \rho < 1$ such that $\norm{f(T(t))} \le \rho$ for $0 < t \le t_0$.
		By the expansion of the Neumann series, $\Id - f(T(t))$ is invertible with
 			\begin{equation} \norm{\left[\Id - f(T(t)) \right]^{-1}} \le \frac{1}{1 - \rho} \qquad \text{for } 0 < t \le t_0. \label{eq:pol_inverse_bound}
			\end{equation}
		Factorization of $1 - f$ yields $\Id - f(T(t)) = (\zeta - T(t))g(T(t))$ for some $g \in \IC[z]$. Observe that $\zeta - T(t)$ is invertible for $0 < t < t_0$ and that its inverse is given by $g(T(t))\left[\Id - f(T(t)) \right]^{-1}$. The boundedness of $T$ on $[0,1]$ in operator norm together with \eqref{eq:pol_inverse_bound} shows that there exists a $K > 0$ such that
			\begin{equation*} \norm{(\zeta - T(t))^{-1}} \le K \qquad \text{for } 0 < t \le t_0. \end{equation*}
		Now, Kato's result (Lemma~\ref{lem:kato}) shows that $(T(t))_{t \ge 0}$ extends to a holomorphic semigroup.
	\end{proof}
	
	\begin{remark}[On Beurling's original proof] 
		A very similar characterization was proven by A. Beurling in ~\cite[Theorem~III]{beu70}. However, the methods used in Beurling's proof completely differ from ours. Whereas our proof is based on Kato's ideas, Beurling's proof is essentially based on ideas originating from \emph{Beurling's analyticity theorem}, a deep theorem giving sufficient conditions for a scalar-valued continuous function defined in an interval to be holomorphic in some rhombus containing the interval. For details on Beurling's result and a statement of the analyticity theorem together with many historical information we refer to J.W. Neuberger's excellent review article~\cite{Neu93}.
	
		We are grateful to Sebastian Król who informed us that part~\ref{thm:beurling_a} of Theorem~\ref{thm:beurling} is also contained in~\cite[Theorem~5.1]{Cas85}. The simple proof we give here allows a generalization to $\mathcal{R}$-analyticity which is the subject of Section~\ref{sec:r-analytic}.
	\end{remark}
	
	Notice that by the maximum principle $\norm{f}_{\mathbb{D}}$ equals $\norm{z \mapsto f(z) z^m}_{\mathbb{D}}$ for every $m \in \IN$, so in particular we obtain the following characterization of holomorphy.
	
	\begin{corollary}[A characterization of holomorphy on the real line] A strongly continuous semigroup $(T(t))_{t \ge 0}$ on a Banach space $X$ is holomorphic if and only if there is $f \in \IC[z]$ such that
		\begin{equation*} \limsup_{t \downarrow 0} \norm{f(T(t))} < \norm{f}_{\mathbb{D}}. \end{equation*}
	\end{corollary}
	
	Moreover, taking $f(z) = (z-1)^N$ for some $N \in \IN$ we obtain a variant of Kato's original result which is invariant under equivalent renorming of the underlying Banach space.
	
	\begin{corollary} Let $(T(t))_{t \ge 0}$ be a $C_0$-semigroup on a Banach space $X$. 
		\begin{thm_enum}
			\item $(T(t))_{t \ge 0}$ extends to a holomorphic semigroup if for some $N \in \IN$
			\begin{equation*} \limsup_{t \downarrow 0} \norm{(T(t) - \Id)^N}^{1/N} < 2. \end{equation*}
			\item Conversely, if $(T(t))_{t \ge 0}$ is holomorphic, there exists $N_0 \in \IN$ such that for every $N \ge N_0$ there exists $K_0 \in \IR_+$ such that
			\begin{equation*} \limsup_{t \downarrow 0} \norm{(T(t) - \Id)^N T(t)^K}^{1/N} < 2  \quad \text{for all } \IN \ni K \ge K_0. \end{equation*}
		\end{thm_enum}
	\end{corollary}
	
	\begin{remark}[On the additional factor of $z^K$] Let $X = H$ be a Hilbert space and $(T(z))_{z \in \Sigma}$ a bounded holomorphic semigroup on $H$ whose negative generator has a bounded $H^{\infty}(\Sigma_{\phi})$-calculus ($\phi \in (0, \pi]$). For the definition of $H^{\infty}$-calculi and their properties we refer to~\cite[Section 9]{KunWei04}. Then by~\cite[Theorem~11.13]{KunWei04} there exists an equivalent scalar product $(\cdot|\cdot)_2$ on $H$ for which $(T(t))_{t \ge 0}$ is contractive. Now~\cite[Corollary~5.8]{Paz83} shows $\limsup_{t \downarrow 0} \norm{T(t) - \Id}_2 < 2$. Hence, for the original norm we obtain
		\begin{equation*} \limsup_{t \downarrow 0} \norm{(T(t) - \Id)^N}^{1/N} < 2 \end{equation*}
	if $N \in \IN$ is chosen sufficiently large. So in this case we can omit the additional factor of $z^K$. We do not know whether this is true in general.
	\end{remark} 
	
	\section{Extrapolation of holomorphic semigroups}\label{sec:holomorphic_extrapolation}
	
	We now use the obtained characterizations of holomorphic semigroups to show that the property of a semigroup to be holomorphic extrapolates. In applications one often considers the following situation. Suppose a strongly continuous holomorphic semigroup $(T_2(z))_{z \in \Sigma}$ on $L^2$ is given together with a consistent semigroup $(T_q(t))_{t \ge 0}$ on $L^q$ for $q \neq 2$, i.e. $T_2(t)f = T_q(t)f$ for all $f \in L^2 \cap L^q$. Then, by the Riesz-Thorin interpolation theorem, $(T_2(t))_{t \ge 0}$ induces consistent strongly continuous semigroups $(T_p(t))_{t \ge 0}$ on $L^p$ for all $p \in (q, 2)$ (or resp. $p \in (2, q)$). The holomorphy of $(T_p(t))_{t \ge 0}$ is then usually deduced from the Stein interpolation theorem, we refer to~\cite[6.2]{Lun09}. Theorem~\ref{thm:beurling} now allows a new natural approach to the problem. We remark that the term semigroup always stands for a not necessarily strongly continuous semigroup, i.e. a mapping $T: [0, \infty) \to \mathcal{L}(X)$ ($X$ Banach space) satisfying $T(0) = \Id$ and $T(t+s) = T(t)T(s)$ for all $t, s \ge 0$. In this way we include the important case $p_2 = \infty$ in the following result.

	\begin{theorem}\label{thm:interpolation_lp_scale} Let $(T_{p_1}(t))_{t \ge 0}$ and $(T_{p_2}(t))_{t \ge 0}$ be two consistent semigroups on $L^{p_1}(\Omega, \mathcal{F}, \mu)$ respectively $L^{p_2}(\Omega, \mathcal{F}, \mu)$ ($1 \le p_1 < p_2 \le \infty$) for a $\sigma$-finite measure space $(\Omega, \mathcal{F}, \mu)$. Suppose that one is strongly continuous and holomorphic and the other is exponentially bounded. Then there exist consistent holomorphic $C_0$-semigroups $(T_p(z))_{z \in \Sigma_p}$ on $L^p(\Omega, \mathcal{F}, \mu)$ for all $p$ strictly between $p_1$ and $p_2$. 	
	\end{theorem}
	\begin{proof} 
		Without loss of generality we may assume that $(T_{p_1}(t))_{t \ge 0}$ is holomorphic. By complex interpolation, the semigroups can be extended to consistent $C_0$-semigroups $(T_p(t))_{t \ge 0}$ on $L^p$ for all $p \in (p_1, p_2)$. Since $(T_{p_1}(t))_{t \ge 0}$ is holomorphic, there exists $f \in \IC[z]$ with $\norm{f}_{\mathbb{D}} = 1$ such that $\limsup_{t \downarrow 0} \norm{f(T_{p_1}(t))} = \rho < 1$. We write $f^N(z) = \sum_{k=0}^{Nn} a_{k,N} z^k$, where $n = \deg f$. Now let $p \in (p_1, p_2)$ be fixed. Observe that $M \coloneqq \sup_{t \in [0,1]}\norm{T_{p_2}(t)} < \infty$. Let $\theta \in (0,1)$ be such that $\frac{1}{p} = \frac{\theta}{p_1} +  \frac{1-\theta}{p_2}$. Then by the Riesz-Thorin interpolation theorem, one has for sufficiently small $t = t(N)$ that
		\begin{align*}
			\norm{f^N(T_p(t))} & \le \norm{f^N(T_{p_1}(t))}^{\theta} \norm{f^N(T_{p_2}(t))}^{1 - \theta} \le \rho^{\theta N} \left( \sum_{k=0}^{Nn} \abs{a_{k, N}} \norm{T_{p_2}(kt)} \right)^{1-\theta}  \\
			& \le M^{1-\theta} \rho^{\theta N} \left( \sum_{k=0}^{Nn} \abs{a_{k,N}} \right)^{1-\theta} = M^{1-\theta} \rho^{\theta N} \left( \sum_{k=0}^{Nn} \frac{1}{k!} \abs{\left(\frac{d}{dx}\right)^{k} f^N(0)} \right)^{1-\theta} \\
			& \le M^{1-\theta} (Nn + 1)^{1-\theta} \rho^{\theta N},
		\end{align*}
		where we have used the usual estimate obtained from Cauchy's integral formula in the last inequality. For $N$ sufficiently large, the right hand side is smaller than $1$ which, by Theorem~\ref{thm:beurling}, yields the holomorphy of $(T_p(t))_{t \ge 0}$.
	\end{proof}
	
	\begin{remark} The idea to use the Kato-Beurling theorem to show extrapolation results for holomorphic semigroups goes back to W. Arendt (see~\cite[Remark~2.7]{FGG+10} or~\cite[7.2.3]{Are04}).
	\end{remark}
	
	One nice feature of the above proof is that it generalizes word by word to more general interpolation methods. We begin with a short overview of the theory of interpolation spaces (see also the detailed expositions in~\cite{BerLoe76}, \cite{Tri78} and~\cite{BruKru91}).
	
	We first define the category of interpolation couples: An \emph{interpolation couple} of Banach spaces is a pair $(X_1, X_2)$ of Banach spaces together with a Hausdorff TVS $\mathcal{X}$ such that $X_1$ and $X_2$ are continuously embedded in $\mathcal{X}$. A morphism between two such couples $(X_1, X_2)$ and $(Y_1, Y_2)$ is a linear operator $T: X_1 + X_2 \to Y_1 + Y_2$ whose restrictions to $X_i$ define bounded linear operators from $X_i$ to $Y_i$, $i= 1,2$. By $\Mor((X_1, X_2), (Y_1, Y_2))$ we denote the set of all morphisms from $(X_1, X_2)$ to $(Y_1, Y_2)$. A Banach space $X$ is called an \emph{interpolation space of exponent} $\theta \in [0,1]$ between $X_1$ and $X_2$ (or with respect to $(X_1, X_2)$) if it satisfies
	\begin{def_enum}
		\item\label{def:interpolation_space_standard_start} $X_1 \cap X_2 \subseteq X \subseteq X_1 + X_2$ (both ends can be endowed with natural norms that make them complete) with continuous embeddings;
		\item If $T \in \Mor((X_1, X_2), (X_1, X_2))$, then $T(X) \subseteq X$ and $T_{|X} \in \mathcal{L}(X)$;
		\item\label{def:interpolation_space_standard_end} For some positive constant $C$ one has for all $T \in \Mor((X_1, X_2), (X_1, X_2))$
			 \[ \norm{T_{|X}}_{\mathcal{L}(X)} \le C \norm{T_{|X_1}}_{\mathcal{L}(X_1)}^{1-\theta} \norm{T_{|X_2}}_{\mathcal{L}(X_2)}^{\theta}; \]
		 \item\label{def:interpolation_space_bonus} For some positive constant $c$ one has
			\begin{equation*}
				\norm{x}_X \le c \norm{x}_{X_1}^{1-\theta} \norm{x}_{X_2}^{\theta} \qquad \text{for all } x \in X_1 \cap X_2.
			\end{equation*} 
	\end{def_enum}
	An interpolation space $X$ between $X_1$ and $X_2$ is called \emph{regular} if $X_1 \cap X_2$ is dense in $X$.
	Note that in the literature an interpolation space of exponent $\theta$ is usually defined as a Banach space satisfying only~\ref{def:interpolation_space_standard_start}-\ref{def:interpolation_space_standard_end}. We add~\ref{def:interpolation_space_bonus} for technical reasons.
	
	An \emph{interpolation functor} $\mathcal{F}$ of \emph{exponent} $\theta \in [0,1]$ is a functor from the category of interpolation couples into the category of Banach spaces such that
	\begin{def_enum}
		\item $\mathcal{F}((X_1, X_2))$ is an interpolation space of exponent $\theta$ with respect to $(X_1, X_2)$ for each interpolation couple $(X_1, X_2)$ and
		\item $\mathcal{F}(T) = T_{|X} \in \mathcal{L}(X, Y)$ for all $T \in \Mor((X_1, X_2), (Y_1, Y_2))$, where $X = \mathcal{F}((X_1, X_2))$ and $Y = \mathcal{F}((Y_1,Y_2))$.
	\end{def_enum}
	
	Moreover, an interpolation functor $\mathcal{F}$ is called \emph{regular} if $\mathcal{F}((X_1, X_2))$ is a regular interpolation space for all interpolation couples $(X_1, X_2)$. Notice that for a space $X = \mathcal{F}((X_1,X_2))$ constructed by an interpolation functor $\mathcal{F}$ of exponent $\theta$ property~\ref{def:interpolation_space_bonus} holds automatically. Indeed, for $x \in X_1 \cap X_2$ apply the functor to the morphism $T: \IC \to X_1 + X_2$ given by $\lambda \mapsto \lambda x$.
	
	We note that the well-known real interpolation (if the interpolation parameter usually denoted by $q$ satisfies $q < \infty$) and complex interpolation methods all define interpolation functors of exponent $\theta \in (0,1)$.
	
	Acquainted with the language of interpolation theory, we obtain the following generalization of Theorem~\ref{thm:interpolation_lp_scale}.
	
	\begin{theorem}\label{thm:interpolation_functor} Let $(T_{1}(t))_{t \ge 0}$ and $(T_{2}(t))_{t \ge 0}$ be consistent semigroups on an interpolation couple $(X_1, X_2)$ and $X$ be a regular interpolation space of exponent $\theta \in (0,1)$ with respect to $(X_1, X_2)$. Assume that one is a holomorphic $C_0$-semigroup and the other is exponentially bounded. Then there exists a unique holomorphic $C_0$-semigroup $(T(z))_{z \in \Sigma}$ on $X$ which is consistent with $(T_1(t))_{t \ge 0}$ and $(T_2(t))_{t \ge 0}$.
	\end{theorem}
	\begin{proof}
		Of course, $(T(t))_{t \ge 0}$ is obtained by interpolation. We only prove that $(T(t))_{t \ge 0}$ is strongly continuous if $(T_1(t))_{t \ge 0}$ or $(T_2(t))_{t \ge 0}$ is. For then, we have
		\begin{equation*}
			\norm{T(t)x - x}_{X} \le c \norm{T_1(t)x - x}_{X_1}^{1-\theta} \norm{T_2(t)x - x}_{X_2}^{\theta} \to 0 \qquad \text{as } t \downarrow 0
		\end{equation*}
		for some constant $c$ and all $x$ in the dense set $X_1 \cap X_2$. This combined with $\sup_{t \in [0,1]} \norm{T(t)}_{\mathcal{L}(X)} < \infty$ yields the strong continuity of $(T(t))_{t \ge 0}$. The remaining part of the proof coincides word by word with the proof of Theorem~\ref{thm:interpolation_lp_scale}.
	\end{proof}

	\section{A Zero-Two Law for cosine families on UMD-Spaces}
	
	As an application of the Kato-Beurling theorem we prove a zero-two law for strongly continuous cosine families. Recall that a strongly continuous mapping $\Cos: \IR \to \mathcal{L}(X)$ for a Banach space $X$ is called a strongly continuous \emph{cosine family} if it satisfies the following properties:
	\begin{def_enum}
		\item $\Cos(0) = \Id$,
		\item $2\Cos(t)\Cos(s) = \Cos(t+s) + \Cos(t-s) \qquad \text{for all } t,s \ge 0$.
	\end{def_enum}
	In this case there exists a uniquely determined operator $A$ called the \emph{generator} of $\Cos$ with $(\omega^2, \infty) \subset \rho(A)$ for some $\omega > 0$ such that
		\begin{equation} \lambda R(\lambda^2, A) = \int_0^{\infty} e^{-\lambda t} \Cos(t) \, dt \qquad \text{for } \lambda > \omega. \end{equation}
	A cosine family has a bounded generator if and only if $\lim_{t \downarrow 0} \norm{\Cos(t) - \Id} = 0$. In that case, $\Cos: \IR \to \mathcal{L}(X)$ is continuous for the operator norm~\cite[Corollary~3.14.9]{ABHN11}.
	
	Cosine families play an important role in the study of abstract second order Cauchy problems (see for example~\cite[Sec. 3.14]{ABHN11} and~\cite[Ch.~2, Sec.~8]{Gol85}). Cosine families can be built systematically from groups: suppose $(U(t))_{t \in \IR}$  is a $C_0$-group. Then $C(t) = \frac{1}{2}(U(t) + U(-t))$ defines a cosine family whose generator is given by the square of the group generator \cite[Example 3.14.15]{ABHN11}.
	
	Various similar zero-one or zero-two laws have been investigated in the past: for a semigroup $(T(t))_{t \ge 0}$ of bounded linear operators $\limsup_{t \downarrow 0} \norm{T(t) - I} < 1$ implies the uniform continuity of $(T(t))_{t \ge 0}$ and the left hand side to be equal to zero. For a strongly continuous group $(U(t))_{t \in \IR}$ the above statement even remains true if $1$ is replaced by $2$. This is indeed a direct consequence of Theorem~\ref{thm:beurling}: $\limsup_{t \downarrow 0} \norm{U(t) - I} < 2$ implies the holomorphy of $(U(t))_{t \in \IR}$ (choose $f(z) = z - 1$) which yields the uniform continuity of $(U(t))_{t \in \IR}$ (cf. proof of Theorem~\ref{thm:zero-two_law}). Similar laws were also investigated in the more general context of Banach algebras. For this as well as the stated results see~\cite{Est04}.
	
	\begin{theorem}[Zero-Two Law]\label{thm:zero-two_law} 
		Let $C(t) = \frac{1}{2}(U(t) + U(-t))$ be the cosine family induced by a strongly continuous group $(U(t))_{t \in \IR}$. Suppose
			\begin{equation} \limsup_{t \downarrow 0} \norm{C(t) - \Id} < 2. \label{eq:cosine_approx} \end{equation}
		Then $(C(t))_{t \in \IR}$ is uniformly continuous and the left hand side of \eqref{eq:cosine_approx} equals $0$.
	\end{theorem}
	\begin{proof}
		Assumption \eqref{eq:cosine_approx} means that $\norm{C(t) - \Id} < \rho$ for some $0 < \rho < 2$ and all sufficiently small $t$, say $0 < t < t_0$. Let $f(z) = \frac{1}{2}(z-1)^2$. Then $\norm{f^N}_{\mathbb{D}} = f^N(-1) = 2^N$ for all $N \in \IN$. Further, we observe that if $M \ge 0$ and $\omega \in \IR$ are chosen such that $\norm{U(t)} \le Me^{\omega t}$ for $t \ge 0$, then
		\begin{align*} 
			\norm{f^N(U(t))} & = \norm{\left(U(t) \left[ \frac{U(t) + U(-t)}{2} - \Id \right] \right)^N} \\
			& \le \norm{U(Nt)} \norm{\left[ \frac{U(t) + U(-t)}{2} - \Id \right]^N} \le M e^{\omega Nt} \rho^N \\
			& \le (M^{1/N} e^{\omega t} \rho)^{N} \le \tilde{\rho}^N < \norm{f^N}_{\mathbb{D}}
		\end{align*}
		for every $0 < \rho < \tilde{\rho} < 2$ provided $N$ is big and $t$ is small enough. Hence, Theorem~\ref{thm:beurling} applied to $f^N$ yields the holomorphy of $(U(t))_{t \in \IR}$. It is well-known that every holomorphic group is even uniformly continuous: $(U(t))_{t \ge 0}$ is immediately norm continuous because it is holomorphic, so
		\begin{equation*}
			U(t) - \Id = U(-1)(U(t+1) - U(1)) \to 0 \qquad \text{in operator norm as } t \downarrow 0. \qedhere
		\end{equation*}
	\end{proof}
	
	In particular, for cosine families on a UMD-space (see~\cite{Bur01} for details) we obtain the following corollary. It solves partially the problem raised by W. Arendt whether a zero-two law does hold for cosine families as well.
	
	\begin{corollary}[Zero-Two Law for Cosine Families on UMD-spaces] Let $\Cos = (C(t))_{t \in \IR}$ be a strongly continuous cosine family on a UMD-space such that \eqref{eq:cosine_approx} holds. Then $\Cos$ is uniformly continuous and the left hand side of \eqref{eq:cosine_approx} equals $0$.
	\end{corollary}
	\begin{proof}
		Let $A$ denote the generator of $\Cos$. Since $A$ is defined on a UMD-space, by Fattorini's theorem there exists a generator $B$ of a $C_0$-group $(U(t))_{t \in \IR}$ and $\omega \ge 0$ such that $A = B^2 + \omega$ \cite[Corollary~3.16.8]{ABHN11}. Let $B(t) = \frac{1}{2}(U(t) + U(-t))$ be the cosine family generated by $B^2$. It is shown in~\cite[Lemma~6.1]{Fat69} that from the iteration given by
			\[ C_0(t) = C(t), \quad C_n(t) = \int_0^t S(t-s) C_{n-1}(s) \, ds \]
		in the strong sense, where $S(t) \coloneqq \int_0^t C(s) \, ds$ is the associated sine function, one obtains $B(t)$ strongly as the series $B(t) = \sum_{n=0}^{\infty} (-\omega)^n C_n(t)$. Moreover, for all $n \in \IN$ one has $\norm{C_n(t)} \le Me^{\omega t} \frac{t^{2n}}{(2n)!}$ for some constants $M \ge 1$ and $\omega \ge 0$. Consequently, we obtain
			\begin{align*}
				\limsup_{t \downarrow 0} \norm{B(t) - \Id} \le \limsup_{t \downarrow 0} \left( \norm{C(t) - \Id} + M e^{\omega t}\sum_{n=1}^{\infty} \frac{(\omega t^{2})^n}{(2n)!} \right) < 2.
			\end{align*}
		Hence, by Theorem~\ref{thm:zero-two_law}, $B^2$ and therefore $A$ are bounded operators which in turn is equivalent to the claim \cite[Corollary~3.14.9]{ABHN11}.
	\end{proof}
	
	\section{The Kato-Beurling theorem for \texorpdfstring{$\mathcal{R}$}{R}-analytic semigroups}\label{sec:r-analytic}
	
	In this section we extend the above method from holomorphic to $\mathcal{R}$-analytic semigroups, a concept which is intimately connected with the problem of maximal regularity. We recall briefly the main concepts, definitions and theorems. 
	
	Let $X$ be a Banach space. A family of operators $\mathcal{T} \subseteq \mathcal{L}(X)$ is called $\mathcal{R}_p$-bounded $(1 \le p < \infty)$ if there exists a finite constant $C_p \ge 0$ such that for each finite subset $\{T_1, \ldots, T_n \}$ of $\mathcal{T}$ and arbitrary $x_1, \ldots, x_n$ one has
		\begin{equation} 
			\norm{\sum_{k = 1}^n r_k T_k x_k}_{L^p([0,1]; X)} \le C_p \norm{\sum_{k=1}^n r_k x_k}_{L^p([0,1]; X)}, \label{eq:R-ineq} 
		\end{equation}
	where $r_k(t) \coloneqq \sign \sin (2^k \pi t)$ denotes the $k$-th \emph{Rademacher function}. The best constant $C_p$ such that \eqref{eq:R-ineq} holds is called the $\mathcal{R}_p$-bound of $\mathcal{T}$ and is denoted by $R_p(\mathcal{T})$. 
	
	Let $\Rad_p(X)$ be the Banach space of sequences $(x_k)_{k \in \IN}$ in $X$ such that $\sum_{k=1}^{\infty} r_k x_k$ converges in $L^p([0,1]; X)$ with the induced norm from $L^p([0,1]; X)$. Observe that a sequence of operators $(T_k)_{k \in \IN} \subset \mathcal{L}(X)$ is $\mathcal{R}_p$-bounded if and only if the map $\mathcal{T}((x_k)_k) = (T_k x_k)_k$ extends to a bounded operator on $\Rad_p(X)$. In this case one has $\norm{\mathcal{T}} = \R[p]{T_n: n \in \IN}$.
	
	The property of being $\mathcal{R}_p$-bounded (but not the constants $C_p$) and therefore the spaces $\Rad_p(X)$ are independent of $p$ by Kahane's inequality \cite[Theorem~2.4]{KunWei04} and therefore the subindex is often omitted (in the following all identities are however only valid for a fixed $p$). The $\mathcal{R}$-bound behaves in many ways similar to a classical norm. For example, if $\mathcal{S}$ is a second family of operators, one sees that (if the operations make sense)
		\[ \mathcal{R}_p(\mathcal{T} + \mathcal{S}) \le \mathcal{R}_p(\mathcal{T}) + \mathcal{R}_p(\mathcal{S}), \qquad \mathcal{R}_p(\mathcal{TS}) \le \mathcal{R}_p(\mathcal{T})\mathcal{R}_p(\mathcal{S}). \]
	Note that a family $\mathcal{T} \subseteq \mathcal{L}(H)$ for some Hilbert space $H$ is $\mathcal{R}$-bounded if and only if $\mathcal{T}$ is bounded in operator norm. Moreover, the Kahane contraction principle~\cite[Propostion~2.5]{KunWei04} says 
		\[ \norm{\sum_{k=1}^n r_k a_k x_k}_{L^p(X)} \le 2 \max \abs{a_k} \norm{\sum_{k=1}^n r_k x_k}_{L^p(X)} \] 
	for $a_1, \ldots, a_n \in \IC$ ($2$ can be replaced by $1$ if $a_1, \ldots, a_n$ are real).
	
	A semigroup $(T(t))_{t \ge 0}$ is called $\mathcal{R}$-analytic if it is holomorphic and there exists a sector $\Sigma_{\delta} \coloneqq \{ z \neq 0: \abs{\arg z} < \delta \} \subset \IC$ ($\delta > 0)$ such that $\R{T(z): z \in \Sigma_{\delta}, \abs{z} \le 1} < \infty$. $\mathcal{R}$-analytic semigroups are studied because of their applications to maximal regularity. One says that the generator $A$ of a semigroup has $\emph{maximal regularity}$ if for one (equivalently: each) $p \in (1, \infty)$ and one $\tau > 0$ (equivalently all $\tau > 0$) the mild solution $x(t) = \int_0^t T(t-s)f(s) \, ds$ of the Cauchy problem $\dot{x} = Ax + f$ with initial condition $x(0) = 0$ is differentiable a.e., $x(t) \in D(A)$ a.e. and $\dot{x}$ and $Ax$ belong to $L^p([0,\tau); X)$. Maximal regularity has important applications in the study of non-autonomous evolution equations and quasi-linear partial differential equations. The connection between the two concepts is the following theorem due to L. Weis~\cite[1.11 \& 1.12]{KunWei04}.
	
	\begin{theorem}[Weis]
		On a UMD-space the generator $A$ of a strongly continuous semigroup $(T(t))_{t \ge 0}$ has maximal regularity if and only if $(T(t))_{t \ge 0}$ is $\mathcal{R}$-analytic.
	\end{theorem}
	
	For a detailed exposition of maximal regularity we refer the reader to \cite{KunWei04} and \cite{DHP03}.
	
	We now generalize Lemma~\ref{lem:kato} and Theorem~\ref{thm:beurling} to the $\mathcal{R}$-analytic case (we show the ``only if''-part for the sake of completeness).
	
	\begin{lemma}\label{lem:kato_R} Let $(T(t))_{t \ge 0}$ be a $C_0$-semigroup. Then $(T(t))_{t \ge 0}$ is an $\mathcal{R}$-analytic semigroup if and only if there exist constants $\abs{\zeta} = 1$, $t_0 > 0$ and $K > 0$ such that $\R{T(t): 0 < t < t_0} < \infty$ and
		\begin{equation*} \zeta \in \rho(T(t)) \text{ for all } 0 < t < t_0 \quad \text{and} \quad \R{(\zeta - T(t))^{-1}: 0 < t < t_0} \le K. \end{equation*}
	In this case the above condition holds for all $\abs{\zeta} \ge 1$, $\zeta \neq 1$.
		
	\end{lemma}
	\begin{proof} 
		First assume that $(T(t))_{t \ge 0}$ satisfies the above condition. Using the same notation as in the proof of Lemma~\ref{lem:kato} (which we use freely without further notice), we obtain by the stronger assumptions in this theorem that for $t\alpha = \theta$ and $\alpha > \frac{\theta}{t_0}$
			\[ (A - i\alpha)^{-1} = -e^{it \alpha}(\zeta - T(t))^{-1} \int_0^t e^{-is \alpha} T(s) \, ds. \]
		Again, making essentially the same estimate as before, we obtain that for $\alpha_1, \ldots, \alpha_n$ with $t_k \alpha_k = \theta$ and $\alpha_k > \frac{\theta}{t_0}$ ($1 \le k \le n$)
			\begin{align*} 
				\R{\alpha_k(A - i\alpha_k)^{-1}: 1 \le k \le n} & \le K \R{\int_0^{t_k} \alpha_k e^{-is \alpha_k} T(s) \, ds} \\
				& = K \R{ \int_0^{\theta} e^{-is} T\left(\frac{s}{\alpha_k} \right) \, ds }.
			\end{align*}
		Hence, using the fact that the $\mathcal{R}$-bound on the right hand side is finite as the strong integral of an $L^1$-function with an $\mathcal{R}$-bounded set~\cite[Corollary~2.14]{KunWei04}, we obtain
		\[ \R{ \alpha(A - i\alpha)^{-1}: \alpha > \frac{\theta}{t_0}} \le K\theta \R{T(t): 0 < t < t_0}. \]
		As before, the same works for negative $\alpha$. Thus there exists an $\alpha_0 \ge 0$ such that $\R{\alpha(A - i\alpha)^{-1}: \abs{\alpha} > \alpha_0} < \infty$. By \cite[Theorem~2.20]{KunWei04}, $(T(t))_{t \ge 0}$ is $\mathcal{R}$-analytic.
		
		Conversely, let $(T(z))_{z \in \Sigma}$ be $\mathcal{R}$-analytic. Then $\{(\lambda - \omega) R(\lambda, A): \lambda \in \Sigma_{\sigma} \}$ is $\mathcal{R}$-bounded for some $\omega$ and $\sigma > \frac{\pi}{2}$. Let $\abs{\zeta} \ge 1$, $\zeta \neq 1$. As shown in Kato's proof in~\cite{Kat70}, one can choose a path $\Gamma$ such that for $d \coloneqq \inf_{z \in \Gamma} \abs{e^z - \zeta} > 0$
			\[ B(t) = \frac{1}{2\pi i} \int_{\Gamma} \frac{e^{z}}{e^{z} - \zeta} t^{-1} R\left(\frac{z}{t}, A\right) \, dz \]
		is a bounded operator for sufficiently small $t \le t_0$. By functional calculus, it is then shown that $R(\zeta, T(t)) = \zeta^{-1} (\Id - B(t))$. In order to finish the proof it therefore remains to show that $\R{B(t): 0 < t < t_0} < \infty$. For $x_1, \ldots, x_n \in X$ and $t_1, \ldots, t_n \in [0, t_0)$ we have (provided $t_0$ is chosen small enough)
			\begin{align*} 
				\MoveEqLeft \norm{\sum_{k=1}^n r_k B(t_k)x_k}_{L^p(X)} \le \frac{1}{2\pi} \int_{\Gamma} \abs{\frac{e^z}{z(e^z - \zeta)}} \norm{\sum_{k=1}^n r_k \frac{z}{t_k} R\left(\frac{z}{t_k}, A\right)x_k}_{L^p(X)} \abs{dz} \\
				& \le \frac{1}{\pi d} \sup_{(t,z) \in [0, t_0] \times \Gamma} \abs{\frac{z}{z - t\omega}} \int_{\Gamma} \abs{\frac{e^z}{z}} \norm{\sum_{k=1}^n r_k \left( \frac{z}{t_k} - \omega \right) R \left( \frac{z}{t_k}, A \right) x_k}_{L^p(X)} \, \abs{dz} \\
				& \le C_1 \int_{\Gamma} \abs{\frac{e^z}{z}} \, \abs{dz} \norm{\sum_{k=1}^n r_k x_k}_{L^p(X)} \le C_2 \norm{\sum_{k=1}^n r_k x_k}_{L^p(X)}. \qedhere
			\end{align*}
	\end{proof}
	
	In the proof of the next theorem we use the following lemma which follows from the Bernstein inequality. Its proof can be found in \cite[p.~398]{beu70}.
	
	\begin{lemma}\label{lem:bernstein_ineq}
		Let $f$ be a trigonometric polynomial of degree $n$ with $\abs{f(x)} \le 1$ for all $x \in \IR$. Then for $N \in \IN$ and $x \in \IR$
			\[ \abs{\left(\frac{d}{dx}\right)^l f^N(x)} \le 	\begin{cases}
											(Nn)^l \abs{f(x)}^{N-l} & \text{if } l \le N, \\
											(Nn)^l & \text{else}.
										\end{cases}
			\]
	\end{lemma}
	
	\begin{theorem}\label{thm:beurling_R} 
		Let $(T(t))_{t \ge 0}$ be a $C_0$-semigroup on a Banach space $X$.
		\begin{thm_enum} 
			\item If $\R{T(t): 0 < t < 1} < \infty$ and there is $f \in \IC[z]$ such that 
			\begin{equation*} \lim_{\epsilon \downarrow 0} \R{f(T(t)): 0 < t < \epsilon} < \norm{f}_{\mathbb{D}}, \end{equation*}
		then $(T(t))_{t \ge 0}$ extends to an $\mathcal{R}$-analytic semigroup. 
			\item 
			Conversely, if $(T(z))_{z \in \Sigma}$ is $\mathcal{R}$-analytic, for each $f \in \IC_1[z]$ there exists $N_0 \in \IN$ such that for every $N \ge N_0$ there exists $K_0 \in \IR_+$ such that
			\[ \lim_{\epsilon \downarrow 0} \R{f^N(T(t))T(Kt): 0 < t < \epsilon} < \norm{f^N}_{\mathbb{D}} \quad \text{for all } K \ge K_0. \]
		\end{thm_enum}
	\end{theorem}
	\begin{proof}
		We again may suppose without loss of generality that $f(\zeta) = 1 = \norm{f}_{\mathbb{D}}$ for some $\zeta \in \partial \mathbb{D}$. Then the theorem can be proven exactly along the lines of the proof of Theorem~\ref{thm:beurling} from which we again borrow the notation: We obtain that
		 	\[ \R{\left[\Id - f(T(t)) \right]^{-1}: 0 < t < t_0} \le \frac{1}{1 - \rho} \] 
		for some $t_0 < 1$ and $0 < \rho < 1$ such that $\R{f(T(t)) : 0 < t < t_0} < \rho$. From this we see by factorization that one has $\R{(\zeta - T(t))^{-1} : 0 < t < t_0} < \infty$. Finally, Lemma~\ref{lem:kato_R} shows the claim.
		
		Conversely, let $z \mapsto T(z)$ be $\mathcal{R}$-analytic in the sector $\Sigma_{\tilde{\delta}}$ and let $ 0 < \delta < \tilde{\delta}$. Note that this implies that for $t \in \IR_+$ the ball around $t$ with radius $t \sin \delta$ is contained in $\Sigma_{\delta}$. Let $R \coloneqq \R{T(z): z \in \Sigma_{\delta}, \abs{z} \le 2}$. For a polynomial $f(z) = \sum_{k=0}^n a_k z^k$ we obtain for $s > 0$ that $f(T(t))T(s) = \sum_{k=0}^n a_k T(s + kt)$. Since $z \mapsto T(z)$ is holomorphic in $\Sigma_{\delta}$, we can write for $0 < nt \le s \sin \delta$
			\begin{align*}
				\sum_{k=0}^n a_k T(s + kt) = \frac{1}{2\pi i} \sum_{l=0}^{\infty} \int_{\abs{z-s} = r} \frac{T(z)}{(z-s)^{l+1}} \, dz \, t^l \sum_{k=0}^n a_k k^l, \qquad (r \le s \sin \delta).
			\end{align*}
		We now replace $f$ by $f^N = \sum_{k=0}^{Nn} a_{k,N} z^k$, where $f \in \IC_1[z]$. After scaling we may again assume that $\norm{f}_{\mathbb{D}} = 1$. Further, we choose $s = Kt$ for $K \in \IR_+$. Then, for $K \ge Nn (\sin \delta)^{-1}$ and $r_t = Kt \sin \delta$ we obtain
			\begin{equation}\label{eq:kato_beurling_proof_1}\begin{split}
				\MoveEqLeft \R{f^N(T(t))T(Kt): t \le K^{-1}} \\
				& \le \sum_{l=0}^{\infty} \abs{\sum_{k=0}^{Nn} a_{k,N} k^l} (2\pi)^{-1} \R{t^l \int_{\abs{z- Kt} = r_t} \frac{T(z)}{(z - Kt)^{l+1}} \, dz} 
			\end{split}\end{equation}
		Moreover,
			\begin{equation}\label{eq:kato_beurling_proof_2}\begin{split}
				\int_{\abs{z- Kt} = r_t} \frac{T(z)}{(z - Kt)^{l+1}} \, dz & = i \int_0^{2\pi} \frac{T(Kt + r_t e^{i \theta})}{(r_t e^{i\theta})^{l+1}} r_t e^{i \theta} \, d\theta \\
				& = i (tK \sin \delta)^{-l} \int_0^{2\pi} T(Kt + r_t e^{i \theta}) e^{-il \theta} \, d\theta.
			\end{split}\end{equation}
		Now, using equations \eqref{eq:kato_beurling_proof_1} and \eqref{eq:kato_beurling_proof_2} together with Lemma~\ref{lem:bernstein_ineq} (for $\tilde{f}(x) = f(e^{ix})$ and $x = 0$) in the second inequality below, we obtain the estimate
			\begin{align*}
				\MoveEqLeft \R{f^N(T(t))T(Kt): t \le K^{-1}} \le R \sum_{l=0}^{\infty} \abs{\sum_{k=0}^{Nn} a_{k,N} k^l} (K \sin \delta)^{-l}  \\
				& \le R \left(\abs{f(1)}^N \sum_{l=0}^N \left(\frac{Nn}{\abs{f(1)} K \sin \delta} \right)^l + \sum_{l = N +1}^{\infty} \left( \frac{Nn}{K \sin \delta} \right)^l \right) \\
				& = R \left( \abs{f(1)}^N \frac{1 - C_{1,K}^{N+1}}{1 - C_{1, K}} + \frac{C_{2,K}^{N+1}}{1 - C_{2,K}} \right),
			\end{align*}
		where $C_{1,K} = \frac{Nn}{\abs{f(1)} K \sin \delta}$ and $C_{2,K} = \frac{Nn}{K \sin \delta}$ (of course for $f(1) = 0$ the first term vanishes). Now, since $\abs{f(1)} < 1$, the right hand side is arbitrarily small provided first $N$ and then $K$ are chosen large enough.
	\end{proof}
	
	Again, we obtain the following characterization via polynomials.
	
	\begin{corollary} Let $(T(t))_{t \ge 0}$ be a strongly continuous semigroup on a Banach space $X$ with $\R{T(t): 0 < t < 1} < \infty$. Then $(T(t))_{t \ge 0}$ is $\mathcal{R}$-analytic if and only if there is $f \in \IC[z]$ such that
		\begin{equation*} \lim_{\epsilon \downarrow 0} \R{f(T(t)) : 0 < t < \epsilon} < \norm{f}_{\mathbb{D}}. \end{equation*}
	\end{corollary}
	
	\section{An extrapolation theorem for \texorpdfstring{$\mathcal{R}$}{R}-analytic semigroups}
	
	Now we use Theorem~\ref{thm:beurling_R} to show an extrapolation theorem for $\mathcal{R}$-analyticity. We formulate the theorem in the abstract context of interpolation spaces in order to emphasize the nature of the used methods. Recall that a Banach space $X$ is called B-convex if it is of non-trivial type or equivalently if $\Rad(X)$ is complemented in $L^2([0,1]; X)$ (see~\cite[Theorem~13.10 \& Theorem~13.15]{DJT95}). Note that every UMD-space is B-convex.

	\begin{theorem}\label{thm:r_analytic_extrapolation} Let $(T_{1}(t))_{t \ge 0}$ and $(T_{2}(t))_{t \ge 0}$ be two consistent semigroups on an interpolation couple $(X_1, X_2)$ of B-convex Banach spaces $X_1$ and $X_2$ and let $X$ be an interpolation space with respect to $(X_1, X_2)$ obtained by a regular interpolation functor $\mathcal{F}$ of exponent $\theta \in (0,1)$. Assume that $(T_1(t))_{t \ge 0}$ is strongly continuous and $\mathcal{R}$-analytic and $\R{T_2(t): 0 < t < 1} < \infty$. Then there exists a unique $\mathcal{R}$-analytic $C_0$-semigroup $(T(z))_{z \in \Sigma}$ on $X$ which is consistent with $(T_1(t))_{t \ge 0}$ and $(T_2(t))_{t \ge 0}$.
	\end{theorem}
	\begin{proof}
		Since $(T_1(z))_{z \in \Sigma}$ is $\mathcal{R}$-analytic, there is $f \in \IC[z]$ with $\norm{f}_{\mathbb{D}} = 1$ and $\deg f = n$ and a constant $\epsilon > 0$ such that $\R{f(T_1(t)) : 0 < t < \epsilon} < \rho < 1$. Moreover, one has $\mathcal{F}((\Rad(X_1), \Rad(X_2))) = \Rad(\mathcal{F}((X_1, X_2)))$ by the B-convexity of $X_1$ and $X_2$~\cite[Remark~6.8]{HHK06}. Now, let $R \coloneqq \R{T_2(t) : 0 < t < 1}$ and $f^N(z) = \sum_{k=0}^{Nn} a_{k,N} z^k$. Then for sufficiently small $\epsilon_0(N)$ one has
		\begin{align*} 
			\MoveEqLeft \R{f^N(T(t)) : 0 < t < \epsilon_0(N)} \le C \R{f^N(T_1(t))}^{1-\theta} \R{f^N(T_2(t))}^{\theta} \\
			& \le C \rho^{N(1 - \theta)} \left(\sum_{k=0}^{Nn} \abs{a_{k,N}} \R{T_2(kt)} \right)^{\theta} \le C R^{\theta} \rho^{N(1-\theta)} \left(\sum_{k=0}^{Nn} \abs{a_{k,N}} \right)^{\theta} \\
			& \le C R^{\theta} (Nn + 1)^{\theta} \rho^{N(1-\theta)}.
		\end{align*}
		The right hand side tends to zero as $N$ tends to infinity. Hence, Theorem~\ref{thm:beurling_R} shows the $\mathcal{R}$-analyticity of $(T(t))_{t \ge 0}$. 
	\end{proof}
		
	
	The by far most important application of Theorem~\ref{thm:beurling_R} is the following corollary which was first proven by W. Arendt \& S. Bu~\cite[Theorem~4.3]{AreBu03}.
	
	\begin{corollary}\label{cor:r_analytic_extrapolation} Let $(T_2(z))_{z \in \Sigma}$ be a holomorphic $C_0$-semigroup on $L^2(\Omega, \mathcal{F}, \mu)$ for some $\sigma$-finite measure space $(\Omega, \mathcal{F}, \mu)$ and $(T_p(t))_{t \ge 0}$ a consistent semigroup on $L^p(\Omega, \mathcal{F}, \mu)$ for $p \neq 2$. If $\R{T_p(t): 0 < t < 1} < \infty$, then the semigroups $(T_q(t))_{t \ge 0}$ obtained by interpolation can be extended to $\mathcal{R}$-analytic $C_0$-semigroups on $L^q(\Omega, \mathcal{F}, \mu)$ for all $q$ strictly between $2$ and $p$. 
	\end{corollary}
	
	\begin{remark}(On Arendt \& Bu's proof) In~\cite[Theorem~3.6]{AreBu03} the authors show that $\mathcal{R}$-analyticity can be characterized by the holomorphy of an associated semigroup on $\Rad(X)$ for which the classical Stein interpolation theorem applies. Our argument again works for more general interpolation functors.
	\end{remark}
	
	Exactly as in \cite[Corollary~4.5]{AreBu03}, the above corollary can directly be applied to semigroups with Gaussian estimates. For the sake of completeness we repeat the main notions and arguments.
	
	Let $\Omega \subseteq \IR^N$ be an open set. A semigroup $(T_p(t))_{t \ge 0}$ on $L^p(\Omega)$ has \emph{Gaussian estimates} if there exist constants $C > 0$ and $a > 0$ such that for all $f \in L^p(\Omega)$
		\begin{equation*} \
			\latticenorm{T_p(t)f}(x) \le C G_p(at)\latticenorm{f}(x) \quad \text{for almost all } x \in \Omega \text{ and all } 0 < t \le 1, 
		\end{equation*}
	where $G_p$ denotes the Gaussian semigroup
	\begin{equation*}
		G_p(t)f = k_t * f \quad (f \in L^p(\IR^N)) \qquad \text{with } k_t(x) = \frac{1}{(4\pi t)^{N/2}} e^{-\abs{x}^2/4t}.
	\end{equation*}
	on $L^p(\IR^N)$. Moreover, given a $\sigma$-finite measure space $(\Omega, \mathcal{F}, \mu)$, the following equivalent characterization of the $\mathcal{R}$-boundedness of a family $(T_n)_{n \in \IN} \subset \mathcal{L}(L^p(\Omega, \mathcal{F}, \mu))$ (for $1 \le p < \infty$) is often useful: The \emph{square function estimate}
		\begin{equation*}
			\norm{\Big(\sum_{k=1}^n \abs{T_k f_k}^2 \Big)^{1/2}}_{L^p} \le C \norm{\Big(\sum_{k=1}^n \abs{f_k}^2 \Big)^{1/2}}_{L^p}
		\end{equation*}
	holds for some $C > 0$, all $n \in \IN$ and all sequences $(f_k)_{k=1}^n \subset L^p(\Omega, \mathcal{F}, \mu)$. This is a consequence of Fubini's theorem together with Kahane's inequality~\cite[Remark~2.9]{KunWei04}.

	\begin{corollary} Let $\Omega \subseteq \IR^N$ be an open set and let $(T_p(t))_{t \ge 0}$ be consistent $C_0$-semigroups on $L^p(\Omega)$ for $1 < p < \infty$. Assume that $(T_2(t))_{t \ge 0}$ is holomorphic and has Gaussian estimates. Then $(T_p(t))_{t \ge 0}$ is $\mathcal{R}$-analytic for all $1 < p < \infty$.
	\end{corollary}
	\begin{proof}
		By Corollary~\ref{cor:r_analytic_extrapolation} and the above remarks, it is sufficient to show square function estimates for the Gaussian semigroup. For this notice that by~\cite[p.~24, Proposition]{Ste93} one has
			\begin{equation*} 
				\sup_{0 < t < 1} \abs{k_t * f}(x) \le \sup_{t > 0} t^{-n/2} \abs{k_t * f}(x) \le c Mf(x) \qquad \text{for all } x \in \IR^N
			\end{equation*}
		for some constant $c \ge 0$ only depending on $k_1$. Hence, the boundedness of the vector-valued maximal operator~\cite[p.~51, Theorem~1]{Ste93} yields for all $n \in \IN$, $t_1, \ldots, t_n \in (0,1)$ and $f_1, \ldots, f_n \in L^p(\IR^N)$
			\begin{equation*}
				\norm{\Big(\sum_{k=1}^n \abs{G_p(t_k) f_k}^2 \Big)^{1/2}}_{p} \le c \norm{\Big(\sum_{k=1}^n \abs{M f_k}^2\Big)^{1/2}}_{p} \le C \norm{\Big(\sum_{k=1}^n \abs{f_k}^2 \Big)^{1/2}}_{p}. \qedhere
			\end{equation*}
	\end{proof}
	
	\begin{remark} Of course, more generally, a variant of the above argument works on arbitrary $\sigma$-finite $L^p$-spaces if there exists an $\mathcal{R}$-bounded family of operators $(S_p(t))_{0 < t \le 1}$ that dominates $(T_p(t))_{0 < t \le 1}$.   
	\end{remark}
	
	\section{Some Remarks on Dilation Arguments}
	
	Dilation theorems for semigroups have been a powerful tool to show maximal regularity for broad classes of semigroups from the beginning. Already the first positive result for maximal regularity by D. Lamberton~\cite[Théorème~1]{Lam87} used dilation arguments. We say that a $C_0$-semigroup $(T(t))_{t \ge 0}$ on a Banach space $X$ \emph{dilates to} a $C_0$-group $(U(t))_{t \in \IR}$ on a second Banach space $Z$ if there exists an isomorphic embedding $J: X \hookrightarrow Z$ such that
		\begin{equation*}
			JT(t) = PU(t)J \qquad \text{for all } t > 0,
		\end{equation*}
	where $P: Z \to J(X)$ is a projection onto $J(X)$.
	
	We want to show that in principle dilation arguments can sometimes be used together with the above results to obtain extrapolation results for $\mathcal{R}$-analyticity. However, we will see that in theses cases there is a more natural and direct way to obtain $\mathcal{R}$-analyticity.
	
	First notice that one has $\R{T(t): 0 < t < 1} < \infty$ if the same holds for the dilated group $(U(t))_{t \in \IR}$. However, one knows that on broad classes of Banach spaces, in particular for $L^p$-spaces $(1 < p < \infty)$, a $C_0$-group is $\mathcal{R}$-bounded on the real line if and only if its generator is an operator of scalar type~\cite[Corollary~7.6 \& Note]{Wei06} which are rare on non-Hilbert spaces. Nevertheless this strategy works for a positive contractive holomorphic $C_0$-semigroup $(T(t))_{t \ge 0}$ on $L^p$ $(1 < p < \infty)$. Such semigroups have a bounded $H^{\infty}(\Sigma_{\phi})$-calculus for some $\phi \in (0, \frac{\pi}{2})$~\cite[4.7.5]{Are04}. Hence, by~\cite[Theorem~5.1]{FroWei06} $(T_p(t))_{t \ge 0}$ dilates to a $C_0$-group $(U(t))_{t \in \IR}$ of isometries on $L^p([0,1];L^p)$ whose generator is of scalar type and one obtains the desired $\mathcal{R}$-boundedness result.
	
	However, one knows that on $L^p$-spaces ($1 < p < \infty)$ a bounded $H^{\infty}(\Sigma_{\phi})$-calculus for some $\phi \in (0, \frac{\pi}{2})$ directly implies $\mathcal{R}$-analyticity~\cite[Remark 12.9c]{KunWei04} and that the semigroups which dilate to a group whose generator is of scalar type are exactly those with such a calculus~\cite[Theorem~5.1]{FroWei06}. So in this case one can always deduce $\mathcal{R}$-analyticity directly without our methods.

	\bibliographystyle{amsalpha}
	\bibliography{approx_identity_holomorphic}{}

\end{document}